\documentclass[12pt]{amsart}
\usepackage{amsfonts,amssymb,amscd,xy}
\usepackage[leqno]{amsmath}
\usepackage{mathrsfs}
\usepackage{amssymb}
\usepackage{amsmath}
\usepackage{amsxtra}
\usepackage{amsthm}
\usepackage[mathscr]{eucal} 
\usepackage{epsfig,amssymb}
\usepackage{bm} 
 
\def\e{\kern 0.08em}
\def\be{\kern -.1em}
\def\le{\kern 0.03em}
\def\lbe{\kern -.025em}
\def\g{\varGamma}

\newcommand{\ksep}{k^{\le\rm sep}}
\newcommand{\kbar}{\overline{k}}
\newcommand{\abar}{\mkern 3.5mu\overline{\mkern-3.5mu A}}
\newcommand{\xbar}{\mkern 3mu\overline{\mkern-3mu X}}
\newcommand{\spec}{\mathrm{ Spec}\,}
\newcommand{\Hom}{ \mathrm{Hom}} 
\newcommand{\s}{\mathscr}
\newcommand{\mr}{\mathrm}

\numberwithin{equation}{section}

\newtheorem{lemma}[equation]{Lemma} 
\newtheorem{theorem}[equation]{Theorem}
\newtheorem{proposition}[equation]{Proposition}
\theoremstyle{definition}
\newtheorem{definition}[equation]{Definition}
\theoremstyle{remark}
\newtheorem{remark}[equation]{Remark}
\newtheorem{remarks}[equation]{Remarks}

\begin{document}

\input xy     
\xyoption{all}

\title[Connected components and Weil restriction] {Galois sets of connected  components and Weil restriction}

\author{Alessandra Bertapelle
\and Cristian D. Gonz\'alez-Avil\'es }
 \thanks{A. Bertapelle,  Dipartimento di Matematica, via Trieste 63, I-35121 Padova, Italy, Email: alessandra.bertapelle@unipd.it}
 \thanks{C. D.  Gonz\'alez-Avil\'es, Departamento de Matem\'aticas, Universidad de La Serena, Cisternas 1200, La Serena 1700000, Chile, Email: cgonzalez@userena.cl} \date{\today}



\begin{abstract}
Let $k$ be a field, $A$ a finite $k$-algebra and $X$ a smooth $A$-scheme.
We describe the Galois set of connected components of the Weil restriction $\Re_{A\lbe/\lbe k}\lbe(X)$ in terms of the sets of connected components of the geometric fibers of $X$. 
\end{abstract}

\maketitle

\section{Introduction}

Let $k$ be a field with fixed algebraic closure $\kbar$, let $\ksep$ be the separable closure of $k$ in $\kbar$ and set $\g=\mr{Aut}_{k}\!\left(\e\kbar\e\right)=\mr{Aut}_{k}(\ksep)$. Further, let $A$ be a finite $k$-algebra and set
\begin{equation}\label{es}
\mathcal S=\left\{\spec h\colon h\in\mr{Hom}_{\e k{\text -\rm alg}}\!\left(A,\kbar\,\right)\right\}.
\end{equation}
For every $A$-scheme $X$ and every $s\in\mathcal S$, we will write $X_{\lbe s}=X\!\times_{\spec A}\!(\spec\kbar,s)$, where $(\spec\kbar,s)$ denotes $\spec\kbar$ regarded as an $A$-scheme via $s$.

The object of this note is to prove the following statement.

\begin{theorem}\label{main0} Let $X$ be a smooth $A$-scheme. Then $\Re_{A\lbe/k}(X)$ is a smooth algebraic space over $k$ which is isomorphic to a scheme over a suitable finite and separable extension of $k$. Further, there exists a canonical isomorphism of $\g$-sets 
\begin{equation*}\label{pi0-p0}
\pi_{0}\!\be\left(\le\Re_{A\lbe/k}(X)\right)\stackrel{\sim}{\to}\prod_{s\e\in\e \mathcal S} \pi_{0}(X_{\lbe s}),
\end{equation*}
where $\mathcal S$ is the set {\rm \eqref{es}}. 
\end{theorem}

\section{Preliminaries on Weil restriction}
Let $f\colon S^{\e\prime}\to S$ be a morphism of schemes and let $X^{\prime}$ be a presheaf of sets on the category $(\mathrm{Sch}/S^{\e\prime}\le)$ of $S^{\e\prime}$-schemes. Then the direct image presheaf
$f_{*}X^{\prime}$, defined by $T\mapsto  X^{\le\prime}(T\lbe\times_{S}\lbe S^{\e\prime}\le)$, is a presheaf of sets on $(\mathrm{Sch}/S\e)$. Further, by \cite[\S 7.6, Lemma 1, p.~191]{blr}, for every $S$-scheme $T$ there exists a canonical bijection (of sets of morphisms of presheaves of sets)
\begin{equation}\label{wr}
\Hom_{\le S}\e(T,f_*X^{\le\prime}\e)\overset{\!\sim}{\to}\Hom_{
S^{\le\prime}}(T\!\times_{S}\!S^{\e\prime},X^{\le\prime}\e).
\end{equation} 
If $X^{\prime}$ is represented by an $S^{\le\prime}$-scheme (which we will also denote by  $X^{\prime}\e$), the presheaf $f_{*}X^{\prime}$ will be denoted by $\Re_{S^{\le\prime}\be/S}(X^{\prime})$. If the latter presheaf is represented by an $S$-scheme, denoted by $\Re_{S^{\le\prime}\be/S}(X^{\prime})$, then we will say that {\it the Weil restriction of $X^{\prime}$ along $f$ exists}.  If $S=\spec A$ and $S^{\e\prime}=\spec A^{\prime}$, we will write $\Re_{A^{\prime}\lbe/\lbe A}$ for $\Re_{S^{\prime}\be/S}$. We refer the reader to \cite[\S 7.6]{blr} and \cite[Appendix A.5]{cgp} for more information on the Weil restriction functor.

In general, to guarantee the existence of the Weil restriction of an $S^{\le\prime}$-scheme $X^{\prime}$ along $f$, it is necessary to impose appropriate conditions on both $f$ and $X^{\prime}$ \cite[\S 7.6, Theorem 4, p.~ 194]{blr}. By a minor modification of the proof of \cite[\S 7.6, Theorem 4, p.~ 194]{blr}, we obtained in \cite[Corollary 2.16]{bga} the following statement.

\begin{proposition}\label{wr-uh} Let $f\colon S^{\e\prime}\to S$ be a finite and locally free {\rm universal homeomorphism}. Then $\Re_{S^{\le\prime}\be/S}(X^{\prime}\e)$ exists for {\rm every} $S^{\le\prime}$-scheme $X^{\prime}$. Further, if $\mathcal U$ is an arbitrary  open  covering of $X^{\prime}$, then $\{\Re_{S^{\le\prime}\be/\lbe S}(U\le)\}_{U\le\in\e\mathcal U}$ is an open  covering of $\Re_{S^{\le\prime}\be/S}(X^{\prime}\e)$.\qed
\end{proposition} 

\begin{remark}\label{rem-wr} Note that the Weil restriction of a non-empty scheme may exist but be the empty scheme. For example, let $A$ be a local $k$-algebra with residue field $k$. Then the canonical closed immersion $\spec k\to \spec A$ induces a closed immersion $\Re_{A/k}(\spec k)\to \Re_{A/k}(\spec A)=\spec k$ \cite[\S 7.6, Proposition 2(ii), p.~192]{blr}. Now, by \eqref{wr},
\[
\mr{Hom}_k(\spec k,\Re_{A/k}(\spec k))\simeq  \mr{Hom}_{A}(\spec A, \spec k)\simeq \mr{Hom}_{A{\text -\rm alg}}(k,A).
\]
Since there exist no $A$-algebra homomorphisms $k\to A$ when $k$ is regarded as an $A$-algebra via the canonical projection $A\to k$, we conclude that $\Re_{A/k}(\spec k)=\emptyset$. 
\end{remark}

As is well-known, in some cases the Weil restriction of a scheme exists only in the larger category of algebraic spaces. For example:

\begin{lemma}\label{algs} Let $X$ be a smooth $A$-scheme. Then $\Re_{A\lbe/\lbe k}(X)$ is a smooth  algebraic space over $k$ which is isomorphic to a smooth scheme over a suitable finite and separable extension of $k$.
\end{lemma}
\begin{proof} Let $k^{\e\prime}$ be an algebraic extension of $k$ such that all residue fields of $A^{\e\prime}=A\be\otimes_{\e k}\be k^{\e\prime}$ are purely inseparable (possibly trivial) extensions of $k^{\e\prime}$. This is the case, for example, if $k^{\e\prime}$ is either $\ksep$ or $\kbar$. Now recall that $A^{\prime}\simeq \prod_{\le m\e\in\e \s M^{\prime}} A_{m}^{\prime}$, where $\s M^{\le\prime}$ denotes the set of maximal ideals of $A^{\prime}$ \cite[Theorem 4.2(b), p.~33]{mil}.  Set $X^{\prime}=X\times_{\spec A}\spec  A^{\prime}$ and $X^{\prime}_{\!A_{m}^{\prime}}=X\times_{\spec A}\spec  A_{m}^{\prime}
\simeq X^{\prime}\times_{\spec A^{\prime}}\spec  A_{m}^{\prime}$. We claim that $\Re_{ A^{\prime}/k^{\prime}}( X^{\prime})$ exists in the category of smooth $k^{\e\prime}$-schemes. Indeed, since $\spec A^{\prime}\simeq \coprod_{\le m\e\in\e \s M^{\le\prime}} \spec   A^{\prime}_{m}$, we have 
\begin{equation}\label{eq-r}  
\Re_{ A^{\prime}\be/  k^{\le\prime}}(X^{\prime})\simeq \Re_{\!\!\!\displaystyle{\prod_{\le m\e\in\e \s M^{\prime}}}\!\!\! A^{\prime}_{\e m}/ k^{\prime}}\! \be\left(\, \coprod_{\le m\e\in\e \s M^{\prime}}(X^{\prime}\times_{\spec  A^{\prime}}\spec   A^{\prime}_m )\right)\simeq     \prod_{m\in  \s M^{\prime} }\Re_{  A^{\prime}_{m}/  k^{\prime}}( X^{\prime}_{\!A_{m}^{\prime}}),
\end{equation}
where the second isomorphism is immediate from \eqref{wr}. Since each factor $\Re_{  A^{\prime}_m/ k^{\prime}}(X^{\prime}_{A_{m}^{\prime}})$ is a smooth $k^{\e\prime}$-scheme by Proposition \ref{wr-uh} and \cite[\S 7.6, Proposition 5(h), p.~195]{blr}, our claim is proved. Now let $k^{\prime\prime}$ be a finite extension of $k$ inside $\kbar$ such that every homomorphism $h$ in \eqref{es} factors through $k^{\prime\prime}$. 
If $k^{\e\prime}$ is the separable closure of $k$ in $k^{\prime\prime}$, then the preceding argument applied to $k^{\e\prime}$ shows that $\Re_{ A^{\prime}\be/\lbe  k^{\le\prime}}(X^{\prime})\simeq \Re_{A\be/\lbe k}(X)\!\times_{\spec k}\!\spec k^{\e\prime}$ is a smooth $k^{\e\prime}$-scheme.  
Further, $\Re_{A^{\prime}\be/\lbe  k^{\le\prime}}(X^{\prime})$ is quasi-separated by \cite[I, Lemma 2.26, p.~51]{kn}. Set $Y=\Re_{ A\lbe/\lbe  k}(X)$, $U=\Re_{ A^{\prime}\be/\lbe  k^{\le\prime}}(X^{\prime}\le)$,  $U^{\prime}=U\!\times_{\spec k}\spec k^{\e\prime}=Y\!\times_{\spec k}\be \spec(\le k^{\e\prime}\!\otimes_{k}\! k^{\e\prime}\le)$. 
The following diagram commutes
\[
\xymatrix{ U\ar[d]_{\Delta_{\e U\be/k^{\le\prime}}}& \ar[l] U^{\prime}\!\times_{U}\be U^{\e\prime}\ar[r]^(0.35){\sim} \ar[d]^{\Psi^{\prime}} & \ar[r] (U\!\times_{Y}\be U)\times_{\spec k}\spec k^{\e\prime}\ar[d]^{\Psi\times \mathrm{id}_{\spec k^{\e\prime}}}&  U\!\times_{Y}\be U \ar[d]^{\Psi}
\\
U\!\times_{\spec  k^{\prime}}\! U&  U^{\e\prime}\!\times_{\spec  k^{\prime}}\be U^{\e\prime}\ar[l]\ar[r]^(0.35){\sim} &  \ar[r] (U\!\times_{\spec k}\be U)\times_{\spec k}\spec k^{\e\prime} & U \!\times_{\spec k} \be U\e,
}
\]
where $\Delta_{\e U\be/k^{\le\prime}}$ is the diagonal immersion and the maps $\Psi$ and $\Psi^{\e\prime}$ are canonical immersions. Note that the right- and left-hand squares are cartesian by construction and \cite[Proposition 1.4.8, p.~36]{ega1}, respectively. Since 
$\Delta_{\e U\be/k^{\le\prime}}$ is quasi-compact, so also are $\Psi^{\prime}$ and $\Psi$ by \cite[Propositions 6.1.5, (ii) and (iii), and 6.1.6, pp.~291 and 293]{ega1}. Thus the canonical morphism $\mathrm{pr}_{1}\colon \Re_{ A^{\prime}\be/\lbe  k^{\le\prime}}(X^{\prime}\le) \to \Re_{A\lbe/\lbe k}(X)$ satisfies the conditions stated in \cite[\S 8.4, Definition 4, p. 224]{blr}. To complete the proof, recall that those properties of morphisms of schemes that are local for the \'etale topology on both the source and the target, such as smoothness, carry over to the category of algebraic spaces, whence it suffices to check these properties on a representable \'etale covering. For a detailed proof of this fact, see \cite[lines 16--22, p.~225]{blr} and \cite[Remark 34.28.7, Lemma 55.22.1, Definition 55.22.2]{spa}. Since $\Re_{A^{\prime}\be/\lbe  k^{\le\prime}}(X^{\prime}\le)$ is a smooth and representable \'etale covering of $\Re_{A\be/\lbe k}(X)$, the latter is a smooth algebraic space over $k$, as claimed. 
\end{proof}

\section{Proof of Theorem \ref{main0}}

We need

\begin{lemma}\label{pi0-1} Let  $B$ be a finite local $\kbar$-algebra and let $X$ be a smooth $B$-scheme. Then $\Re_{B/\e\kbar}\e(X\le)$ is a smooth $\kbar$-scheme and there exists a canonical bijection
\[
\pi_{0}\le(\Re_{B/\le\kbar}\e(X))\simeq \pi_{0}\le(X_{\lbe s}),
\]
where $s\colon \spec \kbar \to \spec B$ is the unique geometric point of $B$.
\end{lemma}
\begin{proof} Let $f\colon \spec B\to \spec \kbar$ be the structure morphism. By Lemma \ref{algs}, $f_{*}\le X=\Re_{B/\e\kbar}\e
(X)$ is a smooth $\kbar$-scheme.  Let $j_{X}\colon X\to s_{\lbe *}X_{\lbe s}$ be the canonical morphism of presheaves of sets which corresponds to $1_{\lbe X_{\lbe s}}$ via \eqref{wr}. Since $f\!\circ\! s=1_{\le\spec \kbar}$, we may identify $f_{*}(s_{*}X_{\lbe s})$ and $X_{\lbe s}$. Then
\begin{equation}\label{fj}
f_{*}(\e j_{\le X})\colon \Re_{B/\le\kbar}\e(X) \to  X_{\lbe s}
\end{equation}
is a morphism of smooth $k$-schemes which induces a map $\pi_{0}\le(\Re_{B/\le\kbar}\e(X))\to \pi_{0}\le(X_{\lbe s})$. To check that the latter map is indeed a bijection, we proceed as follows. Since $s$ is a universal homeomorphism, the closed immersion $X_{\lbe s}\to X$ is a homeomorphism. Now, for every $i\in \pi_{0}\le(X_{\lbe s})$, let $X_{i}$ be the corresponding connected component of $X$. Then $\{X_{i}\}_{i}$ is an open covering of $X$ whence, by Proposition \ref{wr-uh}, the $k$-scheme $\Re_{B/\le\kbar}\e(X)$ is covered by the open subschemes $\Re_{B/\le\kbar}\e(X_{\lbe i})$. 
Since $f_{*}(\e j_{\le X})$ maps $ \Re_{B/\le\kbar}\e(X_{i})$ into $X_{i,\e s}$, it remains only to check that each $k$-scheme $\Re_{B/\le\kbar}\e(X_{i})$ is connected and non-empty.  This follows from \cite[Proposition A.5.9]{cgp} (note that the quasi-projectivity hypothesis in this reference is only needed to guarantee the existence of the indicated Weil restriction, which in our case follows from Proposition \ref{wr-uh} above).
\end{proof}

\begin{remarks}\label{empty2}\indent
\begin{itemize}
\item[(a)] The map $f_{*}(\e j_{\le X})\!\left(\e\kbar\e\right)$ \eqref{fj} is that which corresponds via \eqref{wr} to the canonical map $X(B)\to X\!\!\left(\e\kbar\e\right)=X_{\lbe s}\!\lbe\left(\e\kbar\e\right)$ induced by composition with $s$.
\item[(b)] Remark \ref{rem-wr} shows that the lemma fails if $X$ is the (non-smooth) $B$-scheme $\spec\kbar$. See also \cite[comment after Proposition A.5.9]{cgp}.  
\end{itemize}
\end{remarks}

Let $A$ be a non-zero finite $k$-algebra, let $\abar=A\otimes_{k}\kbar$ and write ${\rm Max}\!\left(\lbe\abar\e\right)$ for the set of maximal ideals of $\abar$. The set $\mathcal S$ \eqref{es} is in bijection with both $\mr{Hom}_{\e k{\text -\rm alg}}\!\left(A,\kbar\,\right)$ and ${\rm Max}\!\left(\lbe\abar\e\right)$. Below these sets will be identified. For example, in the formula $\abar=\prod_{\e s\e\in\le\mathcal S}\abar_{s}$, the elements of $\mathcal S$ are being regarded as maximal ideals of $\abar$.

\smallskip

Let $X$ be an $A$-scheme, set $\xbar= X\times_{\spec  A}\spec  \abar\simeq X\times_{\spec k}\spec \kbar$ and let $X_{\be\abar_s}= X\times_{\spec   A}\spec  \abar_s\simeq \xbar\times_{\spec {\abar}}\spec  \abar_s$. We now describe the $\g$-action on
$\xbar$ in terms of the $\g$-actions on $\spec \abar$ and the covering $\xbar=\coprod_{\e s\le\in\le\mathcal S} X_{\be\abar_{s}}$. Let $\sigma\in\g$ and write $\sigma$ also for the corresponding $k$-linear automorphism of $\abar=A\otimes_{k}\kbar$. Note that $\sigma$ permutes the elements of $\mathcal S={\rm Max}\!\left(\lbe\abar\e\right)$ and induces an isomorphism of $A$-algebras $\abar_{ s}\simeq \abar_{\e\sigma(s)}$ for every $s\in \mathcal S$. By base change, we obtain a well-defined isomorphism of $A$-schemes $\sigma_{\lbe s}=\sigma_{\be s,\e\xbar}\e\colon X_{\be\abar_{\sigma(s)}}\!\simeq X_{\be\abar_s} $. Then the $k$-automorphism $\sigma_{\lbe\xbar}$ of $\xbar=\coprod_{\e s\le\in\le\mathcal S} X_{\be\abar_{s}}$ associated to $\sigma\in \g$ is that which restricts to the isomorphism $\sigma_{\lbe s}$ on  $X_{\be\abar_{\sigma(s)}}$ for every $s\le\in\le \mathcal S$. The preceding description applies, in particular, to $X=\spec A$.

\begin{remark}\label{rgal} Let the notation be as above and consider  the canonical isomorphism
\begin{equation}\label{iso}
\mathrm{Hom}_{\abar}\!\left(\le\spec\abar, \xbar\e\right)\overset{\sim}{\to} \prod_{s\le\in \le\mathcal S} \mathrm{Hom}_{\le\abar_s}\!\be\left(\le\spec \abar_s, \xbar_{\be\abar_s}\le\right), \quad u\mapsto (u_{s})_{s}.
\end{equation}
If $u\colon\spec\abar\to \xbar$ is an $\abar$-morphism and $\sigma\in\g$, then $u^{\sigma}$ is the $\abar$-morphism $\sigma_{\xbar}\circ u \circ \sigma^{-1}_{\be\spec {\abar}}\e$. On the other hand, $\g$ acts on the set on the right in \eqref{iso} via
$(u_{s})^{\le\sigma}=(\sigma_{\be s}\circ u_{\le\sigma(s)} \circ \sigma^{-1}_{\be s})$. Then \eqref{iso} is clearly a bijection of $\g$-sets.

\smallskip

For $s\le\in\le \mathcal S$, let $X_{s}$ denote the fiber of $X$ over the geometric point $s\colon \spec \kbar \to \spec A$ and note that this fiber can be identified with the fiber of $\xbar_{\!\abar_s}$ over the point $s\colon \spec \kbar \to \spec \abar_{s}$. Now let $\sigma_{\be s}\colon X_{\sigma(s)}\to X_{s}$ denote the morphism on geometric fibers induced by $\sigma_{\be s,\e\xbar}$ and consider the following composition of \eqref{iso} with the map induced by base-changes along the morphisms $s\colon \spec \kbar \to \spec \abar_{s}\e$:
\begin{equation}\label{gamma}
\mathrm{Hom}_{\abar}\!\left(\le\spec\abar, \xbar\e\right)\to \prod_{s\le\in \le\mathcal S} \mathrm{Hom}_{\le\abar_s}\!\be\left(\le\spec \abar_s, \xbar_{\be\abar_s}\le\right)\to \prod_{s\le\in \le\mathcal S}\mathrm{Hom}_{\e\kbar }\!\left(\le\spec \kbar, X_{s}\right). 
\end{equation}
We define a $\g$-action on the right-hand set in \eqref{gamma} by setting $(u_{s})^{\le\sigma}=(\sigma_{\be s}\circ u_{\le\sigma(s)} \circ \sigma^{-1}_{\be s})$. Then the preceding composition is a morphism of $\g$-sets.
\end{remark}

Assume now that $X$ is smooth. It follows from Lemma \ref{algs} that the sheaf $\Re_{ A \be/\lbe  k }(X )\times_{\spec k}\spec \kbar\simeq \Re_{ \abar \be/\e\kbar }(\xbar ) $ is represented by a smooth $\kbar$-scheme. We may therefore consider its set of connected components $\pi_{0}\big(\le\Re_{\abar \be/\e  \kbar }\!\left(\xbar\e\right)\!\big)$. Note that the stated isomorphism enables us to define natural $\g$-actions on $\Re_{\abar \be/\e  \kbar }\!\left(\xbar\e\right)$ and $\pi_{0}\big(\le\Re_{\abar \be/\e  \kbar }\!\left(\xbar\e\right)\!\big)$.

\begin{definition}\label{dfpi} Let $X$ be a smooth $A$-scheme. We define the $\g$-set $\pi_0(\Re_{ A \lbe/\lbe  k }(X ))$ to be the set $\pi_{0}\big(\le\Re_{\abar \be/\e  \kbar }\!\left(\xbar\e\right)\!\big)$ equipped with the natural $\g$-action described above.
\end{definition}

{\it We can now prove Theorem {\rm \ref{main0}}:}

\medskip

The first assertion of the theorem is the content of Lemma \ref{algs}. We now prove the second assertion, regarding $\pi_{0}\!\left(\le\Re_{A\lbe/k}(X)\right)$ as a $\g$-set via Definition \ref{dfpi}. Recall that $\spec \abar=\coprod_{\e s\le\in\le\mathcal S} \spec \abar_s$ and $X_{\be\abar_{s}}= X\times_{\spec  A}\spec \abar_{s}\simeq \xbar \times_{\spec {\abar}}\spec  \abar_{s}$. Now, as in \eqref{eq-r}, Lemma \ref{algs} yields a canonical isomorphism of smooth $\kbar$-schemes
\begin{equation}\label{eq-r2}
\Re_{\bar A/\le\kbar}\!\left(\le\xbar\e\right)\simeq  \prod_{s\le\in\le\mathcal S}\Re_{\abar_{\le s}/\le\kbar}\!\left(\le X _{\be\abar_{\le s}}\right)\lbe .
\end{equation}  
From the above isomorphism, \cite[I, \S 4, Corollary 6.10 p.~126]{dg} and Lemma \ref{pi0-1}, we conclude that there exist canonical bijections
\begin{equation*}
\pi_{0}\!\be\left(\Re_{\bar A/\le\kbar}\!\left(\le\xbar\e\right)\right)\simeq \pi_{0}\be\!\left(\,  \prod_{s\le\in\le\mathcal S}\Re_{\abar_{\le s}/\le\kbar}\!\left(\le X _{\be\abar_{\le s}}\right)\right) \simeq\prod_{s\le\in\le\mathcal S}\pi_{0}\!\left(\Re_{\abar_{\le s}/\le\kbar}\!\left(\le X _{\be\abar_{\le s}}\right) \right)\simeq \prod_{s\le\in\le\mathcal S} \pi_{0}(X_{s}).
\end{equation*} 
Let  
\[
\psi\colon \pi_{0}\!\be\left(\Re_{\bar A/\le\kbar}\!\left(\le\xbar\e\right)\right)  \stackrel{\!\sim}{\to}  \prod_{s\le\in\le \mathcal S} \pi_{0}(X_{s}) 
\]
be their composition. It remains only to check that $\psi$ is $\g$-equivariant, where the $\g$-action on $\prod_{\e s\le\in\le \mathcal S} \be\pi_{0}(X_{s})$ is induced by the  $\g$-action on $\prod_{\e s\le\in\le \mathcal S}\be X_{\lbe s}\be\!\left(\e\kbar\e\right)$ defined in Remark \ref{rgal} (by the density of $\kbar$-rational points on smooth $\kbar$-schemes). By construction, $\psi$ is induced by the composition of \eqref{eq-r2} and the product of the canonical morphisms $\Re_{\abar_{s}\be/\e\kbar}\!\left(\e\xbar\e\right)\to X_{\lbe s}$ intervening in \eqref{fj}.  Consequently, $\psi$ lifts to the composition of the canonical maps
\[
\hat{\psi} \colon \mathrm{Hom}_{\abar}\!\left(\le\spec\abar, \xbar\e\right)\overset{\!\sim}{\to} \mathrm{Hom}_{\e\kbar}\!\left(\le\spec \kbar,\Re_{\bar A/\le\kbar}\!\left(\xbar\e\right)\right)\overset{\!\sim}{\to}  \prod_{s\le\in\le\mathcal S}\Re_{\abar_{\le s}/\le\kbar}\!\left(\le X _{\be\abar_{\le s}}\right)\!\left(\e\kbar\e\right)\to \prod_{s\le\in\le\mathcal S}  X_{s}\be\!\left(\e\kbar\e\right), 
\]
where the first isomorphism comes from \eqref{wr} and is $\g$-equivariant, the second isomorphism comes from \eqref{eq-r2} and the third map comes from \eqref{fj}.
Further, it is clear that $\hat \psi$ agrees with the $\g$-equivariant map in \eqref{gamma}, whence $\psi$ is indeed $\g$-equivariant.

\section*{Acknowledgements}
We thank the referee for suggesting the precise form that Theorem \ref{main0} should have and for encouraging us to prove this result. The first author was partially supported by PRAT 2013 ``Arithmetic of varieties over number field'', CPDA 135371/13. The second author was partially supported by Fondecyt grant 1120003.

\end{document}